\documentclass[12pt]{amsproc}
\usepackage{geometry} 
\usepackage{amsmath,amssymb,amsthm}
\usepackage[english]{babel}
\usepackage{tikz}
\usepackage{makecell}
\usepackage{graphicx}
\usepackage{setspace}
\usepackage{array}
\usepackage{multirow}
\usepackage{setspace} 
\usepackage{fancyhdr} 
\usepackage{hyperref}
\hypersetup{hypertex=true,
colorlinks=true,
linkcolor=blue,
anchorcolor=blue,
citecolor=blue}

\title{Sharp pointwise convergence of Schrödinger operator with complex time along curves} 
\author{Binyu Wang, Zhichao Wang}

\newtheorem{theorem}{Theorem}[section]

\newtheorem{proposition}[theorem]{Proposition}

\begin{document}
	\maketitle
\pagestyle{fancy} % 启用 fancy 样式
\fancyhf{} % 清除默认的页眉页尾
\fancyhead[C]{Convergence of Schrödinger operator with complex time along curves }
\fancyhead[R]{\thepage}
	\let\thefootnote\relax\footnotetext{
	This work is supported  in part by the National Natural Science Foundation of China (No.12371100 and No. 12171424)
		}
	\begin{abstract}
		
		In this paper, we study the almost everywhere convergence results of Schrödinger operator with complex time along curves. We also consider the fractional cases. All results are sharp up to the endpoints.
	\end{abstract}
%%%%%%%%%%%%%one%%%%%%%%%%%%%%%%%	
\section{Introduction}
\numberwithin{equation}{section}

The solution to the Schrödinger equation
\begin{equation}\label{0}
	\left\{\begin{array}{ll}
		i u_{t}-\Delta u=0 & (x, t) \in \mathbb{R} \times \mathbb{R}, \\
		u(x, 0)=f(x) & x \in \mathbb{R},
	\end{array}\right.
\end{equation}
can be formally written as
\begin{equation}\label{1}
	e^{i t (-\Delta)} f(x)=\frac{1}{2\pi}\int_{ \mathbb{R}} e^{i\left(x  \xi+t|\xi|^{2}\right)} \widehat{f}(\xi) d \xi.
\end{equation}
Investigating the optimal value   \( s \)  for $f\in H^s\left(\mathbb{R}\right)$ to ensure \( \underset{t \rightarrow 0} {\lim} e^{i t (-\Delta)} f(x)=f(x) \quad a.e. \)  was first put forward by Carleson \cite{MR0576038}, and he proved convergence holding for \( s \geq \frac{1}{4} \), which is optimal according to the counterexample constructed by Dahlberg and Kenig \cite{MR0654188}.

Lee-Rogers \cite{MR2871144} studied the following associated operator
\begin{equation}\label{2}
	U_\Gamma f(x,t)=\frac{1}{2\pi}\int_{ \mathbb{R}}  e^{i\left(\Gamma(x,t)  \xi+t|\xi|^{2}\right)} \widehat{f}(\xi) d \xi.
\end{equation}
Here, \( \Gamma \) is a continuous function such that
\(
\Gamma: \mathbb{R} \times[-1,1] \rightarrow \mathbb{R}, \quad \Gamma(x, 0)=x.
\)
We always assume that the curve $\Gamma(x,t)$ is bilipschitz continuous in \( x \), that is
\begin{equation}\label{3}
	C_{1}\left|x-x^{\prime}\right| \leq\left|\Gamma(x, t)-\Gamma\left(x^{\prime}, t\right)\right| \leq C_{2}\left|x-x^{\prime}\right|, \quad t \in[-1,1], \quad x, x^{\prime} \in \mathbb{R} .
\end{equation}
The curve will be divided into tangential case $(0<\alpha<1)$ and non-tangential case $(\alpha=1)$ according to the Hölder continuous condition \( \alpha \in(0,1] \) in \( t \) , that is
\begin{equation}\label{4}
	\left|\Gamma(x, t)-\Gamma\left(x, t^{\prime}\right)\right| \leq C_{3}\left|t-t^{\prime}\right|^{\alpha}, \quad x \in \mathbb{R}, \quad t, t^{\prime} \in[-1,1].
\end{equation}
	
	When $\Gamma(x,t)$ is bilipschitz  continuous in in \( x \) and Lipschitz  continuous in \( t \), the convergence problem of  (\ref{2}) is essentially equivalent to the traditional convergence problem  of (\ref{1}). In the tangential case, Cho-Lee-Vargas \cite{MR2970037} showed \( s>\max \left\{\frac{1}{4}, \frac{1-2 \alpha}{2}\right\} \) is the sharp sufficient condition.  
	
 Sjölin \cite{MR2549801} first raised   the almost everywhere convergence problem of Schrödinger operator with complex time by replacing $t$ with $t+it^{\gamma}$ in (\ref{1}), that is
 \[
 	P_{\gamma} f(x,t):=\frac{1}{2\pi}\int_{\mathbb{R}} e^{i\left(x \xi+ t|\xi|^{2}\right)} e^{-t^{\gamma}|\xi|^{2}}\hat{f}(\xi) d \xi,
 \] where $\gamma>0$.  And he showed that when $0<\gamma \le 1$, $f\in L^2$ is optimal.  Bailey \cite{MR3047427} pointed that $s>\min\left\{\frac{1}{4},\frac{1}{2}\left(1-\frac{1}{\gamma}\right)^+ \right\}$ is the sharp result.  In addition, some related problems  have also been studied. For example,  Pan--Sun \cite{MR4844992,Pan-Sun}  investigated sequential convergence and  convergence rate of Schrödinger operator with complex time, which were  improved to sharp by Chen--Li--Wang--Wang \cite{Chen-Li-Wang-Wang}.
 
In this paper, we investigate the convergence properties of Schrödinger operator with complex time along curves
\begin{equation}\label{5}
	P_{\gamma} f(\Gamma(x,t),t):=\frac{1}{2\pi}\int_{\mathbb{R}} e^{i\left(\Gamma(x,t)  \xi+ t|\xi|^{2}\right)} e^{-t^{\gamma}|\xi|^{2}}\hat{f}(\xi) d \xi.
\end{equation}
Niu-Xue \cite{MR4052710} proved $s>\min\left\{\frac{1}{2}\left(1-\frac{1}{\gamma}\right)^+ ,\frac{1}{4}\right\}$ is sharp  for $\alpha\in[\frac{1}{2},1]$. We consider the case $\alpha\in(0,\frac{1}{2})$	and simplify the proof of case $\alpha\in[\frac{1}{2},1]$. Here are the main results of our article.

\subsection{Pointwise convergence of Schrödinger operator with complex time along curves.}
\numberwithin{equation}{section}
\begin{theorem}\label{theorem1}(\cite{MR4052710})
	Let $\gamma\in(0,\infty)$. For any curve  $\Gamma(x,t): \mathbb{R}\times[-1,1] \rightarrow \mathbb{R} $ which is bilipschitz continuous in \( x \) and  $\alpha$-Hölder continuous in \( t \) with \(  \frac{1}{2}\leq \alpha \leq 1\),
	we have
	\begin{equation}\label{pointwise convergence 1}
		\lim _{t \rightarrow 0} P_{\gamma} f(\Gamma(x,t),t)=f(x) \text { a.e. } x \in \mathbb{R} ,\quad \forall f \in H^{s}\left(\mathbb{R}\right)
	\end{equation}
	whenever  $ s>s(\gamma)=\min\{\frac{1}{2}(1-\frac{1} {\gamma})^{+}, \frac{1}{4}\}$. Conversely, there exists a curve satisfying the previous assumptions, while (\ref{pointwise convergence 1}) fails whenever $s<s(\gamma)$. 
	
\end{theorem}

\begin{theorem}\label{theorem2}
	Let $\gamma\in(0,\infty)$. For any curve  $\Gamma(x,t) $ which is bilipschitz continuous in \( x \) and   $\alpha$-Hölder continuous in \( t \) with $ 0 < \alpha \leq \frac{1}{4} $, we have convergence result 
	(\ref{pointwise convergence 1}) whenever  $ s>s(\gamma)=\min\{(\frac{1}{2 }-\frac{\alpha} {\gamma})^{+}, 	\frac{1}{2}-\alpha \}$. Conversely, there exists a curve satisfying the previous assumptions, while (\ref{pointwise convergence 1}) fails whenever $s<s(\gamma)$. 
\end{theorem}

\begin{theorem}\label{theorem3}
	Let $\gamma\in(0,\infty)$. For any curve  $\Gamma(x,t) $ which is bilipschitz continuous in \( x \) and   $\alpha$-Hölder continuous in \( t \) with $ \frac{1}{4} < \alpha < \frac{1}{2} $, we have convergence result 
	(\ref{pointwise convergence 1}) whenever  $ s>s(\gamma)$. Conversely, there exists a curve satisfying the previous assumptions, while (\ref{pointwise convergence 1}) fails whenever $s<s(\gamma)$. Here
	\begin{equation}
		s(\gamma)=\begin{cases}
			0             \quad  & \gamma \in(0,2\alpha), \\ 
			\frac{1}{2}-\frac{\alpha} {\gamma}   \quad   & \gamma \in[2\alpha,1),  \\
			\frac{1}{2}-\alpha       \quad   & \gamma \in[1,\frac {1} {2\alpha}),  \\
			\frac{1}{2 }(1-\frac{1} {\gamma})   \quad   & \gamma \in[\frac {1} {2\alpha},2),  \\
			\frac{1}{4}  \quad   & \gamma \in[2,\infty). \
		\end{cases}
	\end{equation}
\end{theorem}

\subsection{Fractional cases.}
~\\
The pointwise  convergence of  fractional Schrödinger operator has also been studied extentsivly. The fractional Schrödinger operator  is defined as
\begin{equation}\label{frac m}
e^{i t (-\Delta)^{\frac{m}{2}}} f(x):=\frac{1}{2\pi}\int_{\mathbb{R}} e^{i \left(x\xi+ t|\xi|^{m}\right)}  \hat{f}(\xi) d \xi
\end{equation}
where $m>0$. Sjölin \cite{MR0904948} and Walther \cite{MR1347033} showed that  $s\ge \frac{1}{4}$ when $m>1$ and $s> \frac{m}{4}$ when $0<m<1$ are sharp results.  When $m=1$,  Rogers-Villarroya \cite{MR2379688} proved that $s>\frac{1}{2}$ is  sufficient enough. The convergence problem of fractional Schrödinger operator along curves has also been analyzed. Cho-Shiraki \cite{MR4307013} and Yuan-Zhao \cite{MR4359958} pointed $s>\max\{ \frac{1}{4},\frac{1-m\alpha}{2}\}$ when $m>1$ and $s>\max\{ \frac{2-m}{4},\frac{1-m\alpha}{2}\}$ when $0<m<1$ can get a.e. convergence.  For  convergence problem of fractional Schrödinger operator with complex time,  Bailey \cite{MR3047427} pointed that $s>\min\left\{\frac{1}{4},\frac{m}{4}\left(1-\frac{1}{\gamma}\right)^+ \right\}$ is the sharp result when $m>1$. Yuan-Zhao-Zheng \cite{MR4222395} extended it to $0<m<1$ and got result $s>\frac{1}{2}\left(1-\frac{1}{\gamma}\right)^+$ when $m=1$. The fractional form of (\ref{5}) is
\[
P^m_{\gamma}f(\Gamma(x,t),t):=\frac{1}{2\pi}\int_{\mathbb{R}} e^{i \left(\Gamma(x,t)+ t|\xi|^{m}\right)} e^{-t^{\gamma}|\xi|^{m}}\hat{f}(\xi) d \xi.
\]

When $m>2$ and $\alpha\in[\frac{1}{m},1]$, Niu-Xue \cite{MR4052710} considered the pointwise convergence problem. We extend the results to $m>0$ and $\alpha\in(0,1]$. Since proof method  is similar to the proof of Theorem \ref{theorem1}, Theorem \ref{theorem2} and Theorem \ref{theorem3}, we will only prensent the results in Section \ref{chap4} while omit the proof.

~\\
\textbf{Outline.} In the following part, we will deal with sufficiency in Section \ref{chap2} and necessity in Section \ref{chap3}. The  fractional cases is presented in Section \ref{chap4}.\\
~\\
\textbf{Notation.} Throughout this article, we write \( A \gg B \) to mean if there is a large constant \( G \), which does not depend on \( A \) and \( B \), such that \( A \geq G B \). We write $A \lesssim B$  to mean that there exists a constant $C$ such that $A \leq C B$. We write $A \gtrsim B$  to mean that there exists a constant $C$ such that $A \geq C B$. We write \( A \sim B \), and mean that \( A \) and \( B \) are comparable.  We write  supp \( \hat{f} \subset\left\{\xi:|\xi| \sim \lambda \right\} \) to mean  supp \( \hat{f} \subset\left\{\xi: \frac{\lambda}{2}\le |\xi| \le 2\lambda \right\} \) and we will always assume $\lambda\gg1$. We write $C_{X}$ to denote a constant that depends on $X$, where $X$ is a variable. We write $A \lesssim_{\varepsilon} B$ to mean, there exists a constant $C_{\varepsilon}$ such that $A \leq C_{\varepsilon} B$.	
	
%%%%%%%%%%%%%two%%%%%%%%%%%%%%%%%%%%%

\section{Proof of upper bounds for maximal functions}\label{chap2}
\numberwithin{equation}{section}
Via a standard smooth approximation argument and  Littlewood-Paley decomposition, Theorem \ref{theorem1}, Theorem \ref{theorem2} and Theorem \ref{theorem3} can be reduced to the following three maximal estimates. These estimates reveal the improved temporal localization phenomenon caused by the addition of term $e^{-t^\gamma \xi^2}$  in oscillatory integration.
\begin{proposition}\label{proposition1}
	Let $\gamma$, $\Gamma$ satisfy assumptions in Theorem \ref{theorem1}. For any $\varepsilon>0$, we have
	\begin{equation}\label{2.1}
	\left\|\sup _{t \in [0,1] }\left|P_{\gamma} f(\Gamma(x,t),t)\right|\right\|_{L^{2}([-1,1])} \lesssim_{\varepsilon} \left\{\begin{matrix}
		\|f\|_{L^{2}} &  \gamma \in(0,1),		\\	\lambda^{\frac{1}{2}(1-\frac{1}{\gamma})+\varepsilon} \|f\|_{L^{2}} &  \gamma \in[1, 2),\\
		\lambda^{\frac{1}{4}} \|f\|_{L^{2}} &\gamma \in[2,\infty),
	\end{matrix}\right.
	\end{equation}
	for all \( f \) with supp \( \hat{f} \subset\left\{\xi:|\xi|\sim \lambda \right\} \). 
\end{proposition}
\begin{proposition}\label{proposition2}
	Let $\gamma$, $\Gamma$ satisfy assumptions in Theorem \ref{theorem2}. For any $\varepsilon>0$, we have
	
	\begin{equation}\label{2.2}
		\left\|\sup _{t \in [0,1] }\left|P_{\gamma} f(\Gamma(x,t),t)\right|\right\|_{L^{2}([-1,1])} \lesssim_{\varepsilon} \left\{\begin{matrix}
		\|f\|_{L^{2}} &  \gamma \in(0,2\alpha),		\\	\lambda^{{\frac{1}{2}}-\frac{\alpha}{\gamma}+\varepsilon} \|f\|_{L^{2}} &  \gamma \in[2 \alpha,1),\\
		\lambda^{\frac{1}{2}-\alpha} \|f\|_{L^{2}} &\gamma \in[1,\infty),
	\end{matrix}\right.
	\end{equation}
	for all \( f \) with supp \( \hat{f} \subset\left\{\xi:|\xi|\sim \lambda \right\} \). 
\end{proposition}
\begin{proposition}\label{proposition3}
	Let $\gamma$, $\Gamma$ satisfy assumptions in Theorem \ref{theorem3}. For any $\varepsilon>0$, we have
	
	\begin{equation}\label{2.3}
	\left\|\sup _{t \in [0,1] }\left|P_{\gamma} f(\Gamma(x,t),t)\right|\right\|_{L^{2}([-1,1])} \lesssim_\varepsilon
	\left\{\begin{matrix}
		\|f\|_{L^{2}} &  \gamma \in(0,2\alpha),		\\	\lambda^{{\frac{1}{2}}-\frac{\alpha}{\gamma}+\varepsilon} \|f\|_{L^{2}} &  \gamma \in[2 \alpha,1),\\
		\lambda^{\frac{1}{2}-\alpha} \|f\|_{L^{2}} &\gamma \in[1,\frac{1}{2\alpha}),\\
		\lambda^{\frac{1}{2}(1-\frac{1}{\gamma})+\varepsilon} \|f\|_{L^{2}}
		&\gamma \in	[\frac{1}{2\alpha},2),\\
		\lambda^\frac{1}{4} \|f\|_{L^{2}}
		&\gamma \in [2,\infty),
	\end{matrix}\right.
\end{equation}
for all \( f \) with supp \( \hat{f} \subset\left\{\xi:|\xi|\sim \lambda \right\} \). 
\end{proposition}

Next we handle the three maximal estimates mentioned above. \\

\textbf{Proof of Proposition \ref{proposition1}, Proposition \ref{proposition2} and Proposition \ref{proposition3}.}\\
 	
 	By linearizing the maximal operator, \( x \mapsto t(x) \) be a measurable function with values in \( [0,1] \), using the $TT^{*}$ method, we only need to consider the kernel, denoted by 
 
  	\[
  	K\left(x, y, t(x), t(y)\right)=\int e^{i\left(\left(\Gamma(x, t(x))-\Gamma\left(y, t(y)\right)\right)  \xi+\left(t(x)-t(y)\right)|\xi|^{2}\right)} e^{-\left(t(x)^\gamma +t(y)^\gamma \right)|\xi|^{2} } \Psi(\frac{\xi}{\lambda}  ) d \xi,
  	\]
  	where \( \Psi \in C_{0}^{\infty}\left(\left(-2,-\frac{1}{2}\right) \cup\left(\frac{1}{2}, 2\right)\right) \) is a positive function.
   \\
  	Next we will prove that for any $\beta_1,\beta_2\geq 0$, $x,y\in[-1,1]$ we have  
   \begin{equation}\label{estimate}
   \left|K\left(x, y, t(x), t(y)\right)\right| \lesssim_{\beta_1,\beta_2}\max\left\{ \min{\left\{   \frac{\lambda^{-2\beta_1}}{\left|x-y\right|^{\frac{\gamma \beta_1}{\alpha}+\frac{1}{2\alpha}}}, \frac{\lambda^{1-2\beta_1}}{{\left|x-y\right|}^{\frac{\gamma \beta_1}{\alpha}}}  \right\}}, \frac{\lambda^{\frac{1}{2}-2\beta_2+\gamma\beta_2}}{|x-y|^{\frac{1}{2}+\gamma\beta_2}} \right\}.
      \end{equation}

  	 Changing  variables \( \xi \rightarrow \lambda \xi \) shows
  	\[
  	K\left(x, y, t(x), t(y)\right) =\lambda \int e^{i\left(\lambda\left(\Gamma(x, t(x))-\Gamma\left(y, t(y)\right)\right)  \xi+\lambda^{2}\left(t(x)-t(y)\right)|\xi|^{2}\right)} e^{-\lambda^2 \left(t(x)^\gamma +t(y)^\gamma \right)|\xi|^{2} } \Psi(\xi ) d \xi.	\]
  	Set
  	\[
  	\varphi(\xi)=\lambda\left(\Gamma(x, t(x))-\Gamma\left(y, t(y)\right)\right)  \xi+\lambda^{2}\left(t(x)-t(y)\right)|\xi|^{2},
  	\]
  	\[
  	G(\xi)=e^{-\lambda^2 \left(t(x)^\gamma +t(y)^\gamma \right)|\xi|^{2} } \Psi(\xi ) ,
  	\]
  	to rewrite $K\left(x, y, t(x), t(y)\right)$ as
  	\[ K\left(x, y, t(x), t(y)\right)=\lambda \int e^{i\varphi(\xi)}G(\xi) d \xi.\] It is evident that $K\left(x, y, t(x), t(y)\right) \lesssim \lambda$.
  	Since \( t(x), t(y) \in [0,1] \), recalling our assumptions on curves, we notice  that
  	\[
  	\begin{aligned}
  		\centering
  		\Gamma(x, t(x))-\Gamma\left(y, t(y)\right) & =\left(\Gamma(x, t(x)-\Gamma\left(x, t(y)\right)\right)+\left(\Gamma\left(x, t(y)\right)-\Gamma\left(y, t(y)\right)\right), \\
  		\left|\Gamma(x, t(x))-\Gamma\left(x, t(y)\right)\right| & \lesssim\left|t(x)-t(y)\right|^{\alpha}, \\
  		\left|\Gamma\left(x, t(y)\right)-\Gamma\left(y, t(y)\right)\right| & \sim|x-y|.
  	\end{aligned}
  	\]
  	
  	We consider two cases \( |x-y| \lesssim\left|t(x)-t(y)\right|^{\alpha},|x-y| \gg\left|t(x)-t(y)\right|^{\alpha} \) separately. 
  	For the first case \( |x-y| \lesssim\left|t(x)-t(y) \right|^{\alpha} \), we have \( \left|\frac{d^{2}}{d \xi^{2}} \varphi\right| \gtrsim \lambda^{2}\left|t(x)-t(y)\right| \). Thus van der Corput's lemma shows	
  	\begin{equation}\label{K2.1}
  		\left|K\left(x, y, t(x), t(y)\right)\right|  \lesssim  \frac{\lambda}{\left(1+\lambda^{2}\left|t(x)-t(y)\right|\right)^{1 / 2}} \left( \sup _{|\xi|\sim 1}|G(\xi)|+\int_{|\xi|\sim 1}\left|G^{\prime}(\xi)\right| d \xi\right) .
  \end{equation}\\  
 Next, we claim that for any $\beta\ge0$, there exists a positive number $C_\beta$ such that
  	\begin{equation}\label{supG}
  		\left( \sup _{|\xi|\sim 1}|G(\xi)|+\int_{|\xi|\sim 1}\left|G^{\prime}(\xi)\right| d \xi\right)\leq  \frac{C_\beta}{\left( \lambda^2 \left|t(x)-t(y)\right|^\gamma \right)^\beta}.
  	\end{equation}			
  	First, observing that $t(x)^{\gamma}+t(y)^{\gamma}\gtrsim\left(t(x)+t(y)\right)^{\gamma} \geqslant \left| t(x)-t(y)\right|^{\gamma}$, and noting that for any \( y>0, \beta\ge0 \), the inequality
  	\[
  	e^{-y} \lesssim_{\beta} y^{-\beta},
  	\]
  	we can get
  	\[
  	\sup _{|\xi|\sim 1}|G(\xi)|\lesssim e^{-\lambda^2 \left|t(x)-t(y)\right|^\gamma } \leq\frac{C_\beta^{'}}{\left( \lambda^2 \left|t(x)-t(y)\right|^\gamma \right)^\beta}.
  	\]
  	As for the second term, we can choose $\beta^{\prime}=\beta+1$ such that
  	\[
  	\begin{aligned}
  		\int_{|\xi|\sim 1}\left|G^{\prime}(\xi)\right| d \xi&\lesssim \int_{|\xi|\sim 1}\sup _{|\xi|\sim 1}|G(\xi)|d\xi+\int_{|\xi|\sim 1} \frac{ C_{\beta^{\prime}}^{\prime}\lambda^{2}\left(t(x)^\gamma +t(y)^\gamma \right)}{\left( \lambda^2 \left(t(x)^\gamma +t(y)^\gamma \right) \right)^{\beta+1}}d\xi\\
  		&\lesssim \frac{C_{\beta}^{\prime}}{\left( \lambda^2 \left|t(x)-t(y)\right|^\gamma \right)^\beta},
  	\end{aligned}
  	\]
  	which completes the proof of (\ref{supG}). \\
  Recalling \( |x-y| \lesssim \left|t(x)-t(y)\right|^{\alpha} \), we combine (\ref{K2.1}) and (\ref{supG}) to get
    		\begin{equation}
  		\begin{aligned}\label{in2.1}
  			\left|K\left(x, y, t(x), t(y)\right)\right|  &\lesssim  \frac{\lambda}{\left(1+\lambda^{2}\left|t(x)-t(y)\right|\right)^{1 / 2}} \left( \sup _{|\xi|\sim 1}|G(\xi)|+\int_{|\xi|\sim 1}\left|G^{\prime}(\xi)\right| d \xi\right) \\
  			&\lesssim \min{\left\{  \frac{1}{\left|t(x)-t(y)\right|^{\frac{1}{2}}}, \lambda  \right\}}
  		 \frac{1}{\lambda^{2\beta_1}\left|t(x)-t(y)\right|^{\gamma \beta_1}}\\
  			& \lesssim \min{\left\{   \frac{\lambda^{-2\beta_1}}{\left|x-y\right|^{\frac{\gamma \beta_1}{\alpha}+\frac{1}{2\alpha}}}, \frac{\lambda^{1-2\beta_1}}{{\left|x-y\right|}^{\frac{\gamma \beta_1}{\alpha}}}  \right\}} .\\
  		\end{aligned}
  	\end{equation}\\
 for any $\beta_1\geq 0$.
 	
    For the second case \( |x-y| \gg\left|t(x)-t(y)\right|^{\alpha} \), we still need to consider two cases 
   \(|x-y| \gg \lambda \left|t(x)-t(y)\right|\) and  \(|x-y| \lesssim \lambda \left|t(x)-t(y)\right| \). If \( |x-y| \gg \lambda\left|t(x)-t(y)\right| \), noticing $\left|\frac{d}{d \xi} \varphi(\xi)\right|\gtrsim \lambda|x-y|$, hence for any $N$ we
   have 
   \[ \left|K\left(x, y, t(x), t(y)\right)\right| \lesssim \frac{\lambda}{(1+\lambda|x-y|)^{N}}\]
by non stationary phase method. If \( |x-y|\lesssim \lambda\left|t(x)-t(y)\right| \), by van der Corput's lemma again for any $\beta_2 \geq 0$, we have
  \begin{equation}\label{in2.2}
\begin{aligned}
\left|K\left(x, y, t(x), t(y)\right)\right| \lesssim &\frac{ \lambda}{\left(1+\lambda^{2}\left|t(x)-t(y)\right|\right)^{1 / 2}}\left( \sup _{|\xi|\sim 1}|G(\xi)|+\int_{|\xi|\sim 1}\left|G^{\prime}(\xi)\right| d \xi\right) \\
\lesssim &\frac{\lambda^{\frac{1}{2}-2\beta_2+\gamma\beta_2}}{|x-y|^{\frac{1}{2}+\gamma\beta_2}}.
\end{aligned}
\end{equation}	\\
  	
Combining estimates (\ref{in2.1}) and (\ref{in2.2}) above, we finish the proof of (\ref{estimate}).\\
  
  	Using  (\ref{estimate}), we have
  	\begin{equation}\label{K}
  		\begin{aligned}
  			I(x)= &\int_{[-1,1]}\left|K(x, y,t(x),t(y))\right| d y \\
  		\leq&\int_{[-1,1]}  \min{\left\{   \frac{1}{\lambda^{2\beta_1}\left|x-y\right|^{\frac{\gamma \beta_1}{\alpha}+\frac{1}{2\alpha}}}, \frac{\lambda^{1-2\beta_1}}{{\left|x-y\right|}^{\frac{\gamma \beta_1}{\alpha}}}  \right\}}+\frac{\lambda^{\frac{1}{2}-2\beta_2+\gamma\beta_2}}{|x-y|^{\frac{1}{2}+\gamma\beta_2}}dy \\
  		    		\lesssim &\int_{\substack{\ |x-y|\lesssim   \lambda^{-2\alpha}   }}\frac{\lambda^{1-2\beta_1}}{\left|x-y\right|^{\frac{\gamma \beta_1}{\alpha}}}  d y+\int_{\substack{\lambda^{-2\alpha} \ll |x-y|\lesssim 1  }}	\frac{\lambda^{-2\beta_1}}{\left|x-y\right|^{\frac{\gamma \beta_1}{\alpha}+\frac{1}{2\alpha}}} d y\\&+\int_{\substack{0 \leq |x-y|\lesssim 1 }} \frac{\lambda^{\frac{1}{2}-2\beta_2+\gamma\beta_2}}{|x-y|^{\frac{1}{2}+\gamma\beta_2}} dy  .   \\
  		%	\lesssim & \lambda^{1-2\alpha+2\beta_1(\gamma-1)}	+\lambda^{\frac{1}{2}-2\beta_2}.
  		\end{aligned}
  	\end{equation}	
  	
 When $\alpha\in [\frac{1}{2},1], \gamma\in[1,2)$, recalling that $K\left(x, y, t(x), t(y)\right)\lesssim \lambda$, we can choose $\beta_1=0$, $\beta_2=\frac{1}{2\gamma}$ and  small enough $\varepsilon>0$  to get
 \[
 I(x) \lesssim_{\varepsilon}   \lambda^{1-2\alpha}+\lambda^{1-\frac{1}{\gamma}+\varepsilon}\lesssim_{\varepsilon}   \lambda^{1-\frac{1}{\gamma}+\varepsilon}.
 \]		
 Therefore we have
 \begin{equation}\label{schurtest}
 	\underset{x \in [-1,1]}{sup } \int\left|K(x, y)\right| d y \lesssim_{\varepsilon} \lambda^{1-\frac{1}{\gamma}+\varepsilon}.
 \end{equation}
 By symmetry we get the same bound for \( \underset{y\in[-1,1]}{\sup} \int\left|K(x, y)\right| d x \). Hence Schur's test gives the required bound (\ref{proposition1}) when $\gamma\in[1,2)$. 
 
 The following table presents under different  pairs of  $(\alpha, \gamma)$, what $(\beta_1,\beta_2)$  we  should choose to  ensure estimates of $I(x)$ are optimal. 
 \begin{table}[!ht]
 \renewcommand{\arraystretch}{1.3}
			\begin{tabular}{m{2.7cm}<{\centering}|m{2cm}<{\centering}|m{2.7cm}<{\centering}|m{2.7cm}<{\centering}|m{3.7cm}<{\centering}}   
				\hline   \textbf{Holder condition $ \alpha$} &\textbf{$\gamma$} & \textbf{$\beta_1$}& \textbf{$\beta_2$} &\textbf{Estimate of $I(x)$}\\     
				\hline
				\multirow{3}*{$\left[\frac{1}{2},1\right]$} & 
				$(0,1)$ &  $0 $ & $\frac{1}{2\gamma}$&$I(x)\lesssim 1$\\ 
			       ~ & $[1,2)$ &  $0$ & $\frac{1}{2\gamma}$ &$I(x)\lesssim_\varepsilon \lambda^{1-\frac{1}{\gamma}+\varepsilon}$ \\
				~ & $[2,\infty)$ & 0 & 0 &$I(x)\lesssim \lambda^{\frac{1}{2}}$\\
				\hline
				\multirow{3}*{$\left(0,\frac{1}{4}\right]$} & $(0,2\alpha)$ &  $\frac{\alpha}{\gamma}$ & $\frac{1}{2\gamma}$ &$I(x)\lesssim 1$ \\ 
								~ & $[2\alpha,1)$ & $\frac{\alpha}{\gamma}$ & $\frac{1}{2\gamma}$ &$I(x)\lesssim_\varepsilon \lambda^{1-\frac{2\alpha}{\gamma}+\varepsilon}$ \\
				~ & $[1,\infty)$ & 0 & 0 &$I(x)\lesssim \lambda^{1-2\alpha}$\\
				\hline
				\multirow{5}*{$\left(\frac{1}{4},\frac{1}{2}\right)$} & $(0,2\alpha)$ &  $\frac{\alpha}{\gamma}$ & $\frac{1}{2\gamma}$ &$I(x)\lesssim 1$\\ 
								~ & $[2\alpha,1)$&  $\frac{\alpha}{\gamma}$ & $\frac{1}{2\gamma}$&$I(x)\lesssim_\varepsilon \lambda^{1-\frac{2\alpha}{\gamma}+\varepsilon}$ \\
				~ & $[1,\frac{1}{2\alpha})$ & $0$ & $\frac{1}{2\gamma}$ &$I(x)\lesssim \lambda^{1-2\alpha}$\\
				~ & $[\frac{1}{2\alpha},2)$ & $0$ & $\frac{1}{2\gamma}$ &$I(x)\lesssim_\varepsilon \lambda^{1-\frac{1}{\gamma}+\varepsilon}$\\
				~ & $[2,\infty)$ & 0 & 0 &$I(x)\lesssim \lambda^{\frac{1}{2}}$\\
				\hline
							\end{tabular}   
		\end{table}

By symmetry we get the same bound for \( \underset{y\in[-1,1]}{\sup} \int\left|K(x, y)\right| d x \).  Thus we finish the proof of Proposition \ref{proposition1}, Proposition \ref{proposition2} and Proposition \ref{proposition3}.

\section{Necessary conditions}\label{chap3}
In order to prove necessity in Theorem \ref{theorem1}, Theorem   \ref{theorem2} and Theorem  \ref{theorem3} , we
use arguments from Niki\v{s}hin-Stein theory, which means we only need to construct counterexamples of maximal functions. Here are the two counterexamples.
\begin{theorem}\label{theorem3.1}
	Let  $\Gamma(x,t)=x-t^{\alpha}$ with $\alpha\in(0,1]$. The following maximal estimate 
		\begin{equation}\label{3.1}
		\left\|\sup _{t \in [0,1] }\left|P_{\gamma} f(\Gamma(x,t),t)\right|\right\|_{L^{2}([-1,1])} \lesssim 
		\left|\left|f\right|\right|_{H^{s}}
	\end{equation}
	 yields 
	 \begin{equation}
	 	s\geq\begin{cases}
	 		\max\left\{\frac{1}{2}-\frac{\alpha} {\gamma},0\right\}   \quad   & \gamma \in(0,1),  \\
	 		\max\left\{\frac{1}{2}-\alpha,0\right\}      \quad   & \gamma \in[1,\infty).\
	 	\end{cases}
	 \end{equation}

\end{theorem}

	\begin{proof}
		Let \( g \in  C_0^{\infty} \) is a standard bump function with supp $g\subset [0,\frac{1}{2}]$ such that \( g(\xi) \geq 0 \) and \( \int g(\xi) d \xi=1 \). We construct a family of functions \( f_{ R} \) with large \( R  \), defined by
		\[
		\widehat{f}_{R}(\eta)=\frac{1}{R} g\left( \frac{\eta}{R}\right).
		\]
		It is easy to calculate the Sobolev norm of ${f}_{R}$
			\[			
		\left\|f_{R}\right\|_{H^{s}} \lesssim R^{s-\frac{1}{2}}.
		\]
		After scaling  $\frac{\eta}{R}\to \xi$, we have
		\[
		\begin{aligned}
	\left|P_{\gamma} f(\Gamma(x,t),t)\right|= &\left|\int e^{i\left((x-t^{\alpha})  \eta+ t|\eta|^{2}\right)} e^{-t^{\gamma}|\eta|^{2}}\frac{1}{R} g\left( \frac{\eta}{R}\right) \frac{d \eta}{2 \pi}\right|\\
	=&\left|\int e^{i\left((x-t^{\alpha})  R \xi+ tR^{2}|\xi|^{2}\right)} e^{-t^{\gamma}R^{2}|\xi|^{2}}g(\xi) \frac{d \xi}{2 \pi}\right|.
\end{aligned}	
\]
Set 
\[
\phi_R(\xi,x,t)=(x-t^{\alpha})R\xi+tR^2\xi^2,\quad \Phi_R(\xi,t)=-t^{\gamma}R^2|\xi|^{2}.
\]
Next we will choose suitable $t$ depending on $x$, which we label $t_x$ to ensure  $\phi_R(\xi,x,t)$ and $\Phi_R(\xi,t)$ considerablely small.
When $\gamma\in(0,1)$, one can find $t^{\gamma}R^2 \geq tR^2$ for $t\in[0,1]$. For some suitably small constant c, let
\[
A=\{x:0 <x<cR^{-\frac{2\alpha}{\gamma}}\}.
\]
Set 
\begin{equation}\label{3.3}
	t_x=x^{\frac{1}{\alpha}},
\end{equation}
 we can find $t_x\in(0,cR^{-\frac{2}{\gamma}})$ to have  $t_xR^{2}|\xi|^2\leq c$ and  $t_x^\gamma R^{2}|\xi|^2\leq c$. Thus
		\begin{equation}\label{3.4}
			\begin{aligned}
			\sup _{t\in[0,1]}\left|P_{\gamma} f(\Gamma(x,t),t)\right|
			& \geq\left|P_{\gamma} f(\Gamma(x,t_x),t_x)\right| \\
			& \geq \left|\int g(\xi)\frac{d \xi}{2 \pi}\right|-\int \left|e^{i\phi_R(\xi,x,t_x)} e^{\Phi_R(\xi,t_x)}-1\right| g(\xi) \frac{d \xi}{2 \pi} \\
			 & \geq \frac{1}{4\pi}.
			\end{aligned}
	\end{equation}
Let $R \to \infty$, (\ref{3.1}) yields\\
		 \begin{equation}\label{3.5}
		 	s\geq\frac{1}{2}-\frac{\alpha}{\gamma}.
		 \end{equation}\\
	 
		When $\gamma\in [1,\infty)$, we can find $t^{\gamma}R^2 \leq tR^2$. Consider \[
		B=\{x:0 <x<cR^{-2\alpha}\}.
		\]Set (\ref{3.3}) as before, we can find  $\phi_R(\xi,x,t)$ and $\Phi_R(\xi,t)$ small enough to get (\ref{3.4}). Let $R \to \infty$, (\ref{3.1}) yields\\
		\begin{equation}\label{3.6}
			s\geq\frac{1}{2}-\alpha.
		\end{equation}\\
	Finally, set $t=0$ and $x<\frac{c}{R}$, we have (\ref{3.4}) as well to get 
	\begin{equation}\label{3.7}
		s\geq0.
	\end{equation}\\
  Combing (\ref{3.5}), (\ref{3.6}) and (\ref{3.7}), we finish the proof.
	\end{proof}

\begin{theorem}\label{theorem3.2}
	Let $\gamma\in[1,\infty)$, $\Gamma(x,t)=x-t^{\alpha}$ with $\alpha\in(\frac{1}{4},1]$. The  maximal estimate (\ref{3.1}) yields  \begin{equation}
		s\geq\begin{cases}
			\frac{1}{2}(1-\frac{1} {\gamma})   \quad   & \gamma \in\left[\max\{\frac{1}{2\alpha},1\},2\right),  \\
			\frac{1}{4}      \quad   & \gamma \in[2,\infty).\
		\end{cases}
	\end{equation}
\end{theorem}
\begin{proof}
	Let  \( g \in  C_0^{\infty} \) is a standard bump function with supp $g\subset [0,\frac{1}{2}]$ such that \( g(\xi) \geq 0 \) and \( \int g(\xi) d \xi=1 \). We construct a family of functions \( f_{ R} \) with large \( R  \), defined by
	\[
	\widehat{f}_{R}(\eta)=\frac{1}{R} g\left( \frac{\eta+R^b}{R}\right).
	\]
		Here $b\in[1,\infty)$. It is easy to calculate the Sobolev norm of ${f}_{R}$
	\[			
	\left\|f_{R}\right\|_{H^{s}} \lesssim R^{bs-\frac{1}{2}}.
	\]
	After scaling  $\frac{\eta+R^b}{R}\to\xi$, we have
	\[
	\begin{aligned}
		\left|P_{\gamma} f(\Gamma(x,t),t)\right|= &\left|\int e^{i\left((x-t^{\alpha})  \eta+ t|\eta|^{2}\right)} e^{-t^{\gamma}|\eta|^{2}}\frac{1}{R} g\left( \frac{\eta+R^b}{R}\right) \frac{d \eta}{2 \pi}\right|\\
		=&\left|\int e^{i\left((x-t^{\alpha})  (R \xi-R^b)+t|R\xi-R^b|^2\right)} e^{-t^{\gamma}|R\xi-R^b|^{2}}g(\xi) \frac{d \xi}{2 \pi}\right|\\
		=&\left|\int e^{i\left((x-t^{\alpha}-2R^bt)  R \xi+t|R\xi|^2\right)} e^{-t^{\gamma}|R\xi-R^b|^{2}}g(\xi) \frac{d \xi}{2 \pi}\right|.
	\end{aligned}	
	\]
	Set\[ \phi_R(\xi,x,t)=(x-t^{\alpha}-2R^bt)R\xi+tR^2\xi^2,    
	   \Phi_R(\xi,t)=-t^{\gamma}|R\xi-R^b|^{2}.\]
	When $\gamma\in[\max\{\frac{1}{2\alpha},1\},2)$, we set $b=\gamma$. For
	  \[x\in
	A=\{x:0 <x<c^\alpha R^{-2\alpha}+2cR^{\gamma-2}\},
	\]
we can find a $t_x\in(0,cR^{-2})$ satisfying 
\[x-t_x^{\alpha}-2R^\gamma t_x=0.\] 
Thus we have  $\phi_R(\xi,x,t_x)$ and $\Phi_R(\xi,t_x)$ small enough to get (\ref{3.4}) as well on $x\in A$.  Notice that $R^{-2\alpha} \leq R^{\gamma-2}$, $|A| \geq 2cR^{\gamma-2}$.  Let $R \to \infty$, (\ref{3.1}) yields
\begin{equation}\label{3.9}
	s\geq \frac{1}{2}\left(1-\frac{1}{\gamma}\right).
\end{equation}\\
	When $\gamma\in[2,\infty)$, we set $b=2$.  For
	\[x\in
	B=\{x:0 <x<c^\alpha R^{-2\alpha}+2c\},
	\]
	we can find a $t_x\in(0,cR^{-2})$ satisfying 
	\[x-t_x^{\alpha}-2R^2 t_x=0.\]
	Thus we have  $\phi_R(\xi,x,t_x)$ and $\Phi_R(\xi,t_x)$ small enough to get (\ref{3.4}) as well on $x\in B$. Notice that $R^{-2\alpha} \leq 1$ and $|B| \geq 2c$.  Let $R \to \infty$, (\ref{3.1}) yields
	\begin{equation}\label{3.10}
	s\geq \frac{1}{4}.
\end{equation}\\
	Combing (\ref{3.9}), (\ref{3.10}), we finish the proof.
	\end{proof}

\section{Convergence  of fractional Schrödinger operator with complex time along curves}\label{chap4}
The similar method can also handle the convergence  of fractional Schrödinger operator with complex time along curves. Here, we only present results while omit the proof.

When $m\in(0,1)$, we have the following two theorems.
\begin{theorem}\label{theorem4.1}
	Let $m\in(0,1)$ and $\gamma\in(0,\infty)$. For any curve  $\Gamma(x,t): \mathbb{R}\times[-1,1] \rightarrow \mathbb{R} $ which is bilipschitz continuous in \( x \) and  $\alpha$-Hölder continuous in \( t \) with \(  0< \alpha \leq \frac{1}{2}\),
	we have
	\begin{equation}\label{pointwise convergence mcass}
		\lim _{t \rightarrow 0} P^m_{\gamma} f(\Gamma(x,t),t)=f(x) \text { a.e. } x \in \mathbb{R} ,\quad \forall f \in H^{s}\left(\mathbb{R}\right)
	\end{equation}
	whenever  $ s>s(\gamma)=\min \left\{\frac{1}{2 }(1-m\alpha), \frac{1}{2 }\left(1-\frac{m\alpha}{\gamma} \right)^{+} \right\} $. Conversely, there exists a curve satisfying the previous assumptions, while (\ref{pointwise convergence mcass}) fails whenever $s<s(\gamma)$. 
\end{theorem}

\begin{theorem}\label{theorem4.2}
	Let $m\in(0,1)$ and $\gamma\in(0,\infty)$. For any curve  $\Gamma(x,t) $ which is bilipschitz continuous in \( x \) and   $\alpha$-Hölder continuous in \( t \) with $ \frac{1}{2} < \alpha \leq 1 $, we have convergence result 
	(\ref{pointwise convergence mcass}) whenever  $ s>s(\gamma)=\min\{\frac{2-m}{4}+\frac{m(1-2\alpha)}{4\gamma} , \frac{1}{2 }\left(1-\frac{m\alpha}{\gamma} \right)^{+} \}$. Conversely, there exists a curve satisfying the previous assumptions, while (\ref{pointwise convergence mcass}) fails whenever $s<s(\gamma)$.
	
\end{theorem}
When $m=1$, we have the following theorem.
\begin{theorem}\label{theorem4.3}
	Let $m=1$ and   $\gamma\in(0,\infty)$. For any curve  $\Gamma(x,t) $ which is bilipschitz continuous in \( x \) and   $\alpha$-Hölder continuous in \( t \) with $ \alpha\in (0,1]$, we have convergence result 
	(\ref{pointwise convergence mcass}) whenever  $  s>s(\gamma)=\frac{1}{2 }\left(1-\frac{\alpha}{\gamma} \right)^{+}$. Conversely, there exists a curve satisfying the previous assumptions, while (\ref{pointwise convergence mcass}) fails whenever $s<s(\gamma)$. 
\end{theorem}
When $m\in(1,\infty)$, we have the following  theorems.
\begin{theorem}\label{theorem4.4}
	Let $m\in(1,\infty)$ and  $\gamma\in(0,\infty)$. For any curve  $\Gamma(x,t) $ which is bilipschitz continuous in \( x \) and   $\alpha$-Hölder continuous in \( t \) with $ 0 < \alpha \leq \frac{1}{2m} $, we have convergence result 
	(\ref{pointwise convergence mcass}) whenever  $ s>s(\gamma)=\min \left\{\frac{1}{2 }(1-m\alpha), \frac{1}{2 }\left(1-\frac{m\alpha}{\gamma} \right)^{+} \right\} $. Conversely, there exists a curve satisfying the previous assumptions, while (\ref{pointwise convergence mcass}) fails whenever $s<s(\gamma)$. 
\end{theorem}
\begin{theorem}\label{theorem4.5}
Let $m\in(1,\infty)$ and  $\gamma\in(0,\infty)$. For any curve  $\Gamma(x,t) $ which is bilipschitz continuous in \( x \) and   $\alpha$-Hölder continuous in \( t \) with $ \frac{1}{2m} < \alpha < \min\{\frac{1}{2},\frac{1}{m} \}$, we have convergence result 
(\ref{pointwise convergence mcass}) whenever  $ s>s(\gamma) $. Conversely, there exists a curve satisfying the previous assumptions, while (\ref{pointwise convergence mcass}) fails whenever $s<s(\gamma)$. Here
	\begin{equation}
		s(\gamma)=\begin{cases}
			0             \quad  & \gamma \in(0,m\alpha), \\ 
			\frac{1}{2 }-\frac{m\alpha} {2\gamma}  \quad   & \gamma \in[m\alpha,1),  \\
			\frac{1-m\alpha}{2}       \quad   & \gamma \in[1,\frac {m} {m-2+2m\alpha}),  \\
			\frac{m}{4}(1-\frac{1} {\gamma})   \quad   & \gamma \in[\frac {m} {m-2+2m\alpha},\frac {m}{m-1}),  \\
			\frac{1}{4}  \quad   & \gamma \in[\frac {m}{m-1},\infty). \
		\end{cases}
	\end{equation}
\end{theorem}

\begin{theorem}\label{theorem4.7}
	Let $m\in(1,\infty)$ and  $\gamma\in(0,\infty)$. For any curve  $\Gamma(x,t) $ which is bilipschitz continuous in \( x \) and   $\alpha$-Hölder continuous in \( t \) with $ \frac{1}{m} \leq \alpha \leq 1$, we have convergence result 
	(\ref{pointwise convergence mcass}) whenever  $ s>s(\gamma)=\min\{\frac{1}{4},\frac{m}{4}(1-\frac{1} {\gamma})^{+}\} $. Conversely, there exists a curve satisfying the previous assumptions, while (\ref{pointwise convergence mcass}) fails whenever $s<s(\gamma)$. 
\end{theorem}

\begin{theorem}\label{theorem4.6}
	Let $m\in(1,2)$ and $\gamma\in(0,\infty)$. For any curve  $\Gamma(x,t) $ which is bilipschitz continuous in \( x \) and   $\alpha$-Hölder continuous in \( t \) with $ \frac{1}{2} \leq \alpha <\frac{1}{m} $, we have convergence result 
	(\ref{pointwise convergence mcass}) whenever  $ s>s(\gamma) $. Conversely, there exists a curve satisfying the previous assumptions, while (\ref{pointwise convergence mcass}) fails whenever $s<s(\gamma)$. Here
	\begin{equation}
		s(\gamma)=\begin{cases}
			0             \quad  & \gamma \in(0,m\alpha), \\ 
			\frac{1}{2 }-\frac{m\alpha} {2\gamma}  \quad   & \gamma \in[m\alpha,1),  \\
			\frac{2-m}{4}+\frac{m(1-2\alpha)}{4\gamma}       \quad   & \gamma \in[1,\frac {m(1-\alpha)} {m-1}),  \\
			\frac{m}{4}(1-\frac{1} {\gamma})   \quad   & \gamma \in[\frac {m(1-\alpha)} {m-1},\frac {m}{m-1}),  \\
			\frac{1}{4}  \quad   & \gamma \in[\frac {m}{m-1},\infty). \
		\end{cases}
	\end{equation}
\end{theorem}

~\\

Binyu Wang, Department of Mathematics, Zhejiang University, Hangzhou 310058,  People’s Republic of China
		
		E-mail address: 12535002@zju.edu.cn
		
		Zhichao Wang, Department of Mathematics, Zhejiang University, Hangzhou 310058,  People’s Republic of China
		
		E-mail address: zhichaowang@zju.edu.cn

\end{document}